\documentclass[11pt]{amsart}

\usepackage{graphicx}
\usepackage{alltt}

\usepackage{graphics,color}

\usepackage{psfrag}

\newcommand{\PSLC}{{\bf {PSL}}(2,\C)}
\newcommand{\D} {\mathbb {D}}

\newcommand{\R} {\mathbb {R} }

\newcommand{\N}{\mathbb {N}}
\newcommand{\C} {\mathbb {C}}

\newcommand{\Sp} {\mathbb {S}}

\newcommand{\pt}{\partial}
\newcommand{\wt}{\widetilde}
\newcommand{\wh}{\widehat}

\newcommand{\Ur} {\mathcal {U}}
\newcommand{\Kr}{{\it {K}} }

\newcommand{\diam}{\operatorname{diam}}

\newcommand{\stab}{\operatorname{stab}}

\newcommand{\id}{\operatorname{Id}}

\newcommand{\Homeo}{\operatorname{Homeo}}

\newcommand{\kernel}{\operatorname{ker}}

\newtheorem{corollary}{Corollary}[section]
\newtheorem{theorem}{Theorem}[section]
\newtheorem{definition}{Definition}[section]
\newtheorem{lemma}{Lemma}[section]
\newtheorem{conjecture}{Conjecture}
\newtheorem{problem}{Problem}[section]
\newtheorem{proposition}{Proposition}[section]
\newtheorem{claim}{Claim}

\theoremstyle{remark}
\newtheorem*{remark}{Remark}

\begin{document}

\title[Criterion for  Cannon's Conjecture] {Criterion for Cannon's Conjecture}
\author[Vlad Markovic]{Vladimir Markovic}

\address{\newline California Institute of Technology \newline Department of Mathematics \newline Pasadena, CA 91125,  USA}
\email{markovic@caltech.edu}

\today


\begin{abstract} The Cannon Conjecture from the geometric group theory asserts that a word hyperbolic group that acts effectively on its boundary, and  whose boundary is homeomorphic to the 2-sphere, is isomorphic to a Kleinian group. We prove the following Criterion for Cannon's Conjecture:
A hyperbolic group $G$ (that acts effectively on its boundary) whose boundary is homeomorphic to the 2-sphere is isomorphic to a Kleinian group if and only if every two points in the boundary of $G$ are separated by a quasi-convex surface subgroup. Thus, the Cannon's conjecture is reduced to showing that such a group contains ``enough" quasi-convex surface subgroups. 
\end{abstract}

\maketitle

%All this just to make a footnote without a number.
\let\johnny\thefootnote
\renewcommand{\thefootnote}{}

\footnotetext{Vladimir Markovic is supported by the NSF grant number DMS-1201463}
\let\thefootnote\johnny

\section{Introduction} 
\subsection{The statement of the Criterion for Cannon's Conjecture} Let $G$ be a hyperbolic group (by this we mean that $G$ is word hyperbolic).  One of the central problems in the geometric group theory is the Cannon' Conjecture:

\begin{conjecture}[Cannon's Conjecture] Let $G$ be a hyperbolic group (that acts effectively on its boundary) whose boundary $\pt{G}$ is homeomorphic to the 2-sphere $\Sp^2$ ($\pt{G} \approx \Sp^2$). Then $G$ is isomorphic to a Kleinian group.  In particular, if  $G$ is torsion-free then it is isomorphic to the fundamental group of a closed hyperbolic 3-manifold.
\end{conjecture}

\begin{remark} Let $G$ denote a hyperbolic group $G$ and let  $\mu:G \to \Homeo(\pt{G})$ be the standard homomorphism. The kernel $\kernel \mu$ is a finite normal subgroup of $G$ (it is finite since $G$ acts as a convergence group on $\pt{G}$).
We say that $G$ acts effectively on its boundary if $\kernel \mu$ is trivial. If $\kernel \mu$ is not trivial we can replace $G$ by the quotient $G_1=G/ \kernel \mu$. The group $G_1$ is a hyperbolic group too and it acts effectively on its boundary.
\end{remark}

The Hyperbolization Theorem (proved by Perelman) states that every closed, irreducible, atoroidal 3-manifold with infinite fundamental group is hyperbolic. Cannon devised a two part program towards proving this theorem. The first part was to prove the Weak Hyperbolization Conjecture that states that the fundamental group of a closed, irreducible and atoroidal 3-manifold with infinite fundamental group is hyperbolic. The second part is to prove the Cannon's Conjecture. Combining these two parts would imply that the fundamental group of a closed, irreducible and atoroidal 3-manifold with infinite fundamental group is isomorphic to a torsion free, co-compact lattice in $\PSLC$, that is, it is isomorphic to the fundamental group of a hyperbolic manifold. The Hyperbolization Theorem would then follow from  the result of Gabai-Meyerhoff-Thurston \cite{g-m-t} which states that 3-manifolds that are homotopy equivalent to hyperbolic 3-manifolds are themselves hyperbolic.

In particular, the positive solution to the Cannon's Conjecture would imply that a closed and irreducible  3-manifold whose fundamental group is word-hyperbolic is  hyperbolic. For example, 
the fundamental group of a negatively curved closed 3-manifold is word-hyperbolic. Thus proving the Cannon's Conjecture would offer a new proof of a major chunk of the Hyperbolization Theorem.

Another well known open problem in Geometric group theory is the following question of Gromov:

\begin{problem}[Gromov] Let $G$ be one-ended hyperbolic group. Does $G$ necessarily contain a surface subgroup?
\end{problem}
We say that a group is a surface group if is isomorphic to the fundamental group of a closed surface $S_g$ of genus  $g \ge 2$. Particularly interesting question is whether $G$ contains a quasi-convex surface subgroup.
Kahn and Markovic  \cite{kahn-markovic} proved the Surface Subgroup Theorem, thus answering the Gromov's question in the positive in the case when $G$ is the fundamental group of a closed hyperbolic 3-manifold. Calegari \cite{calegari} showed that (under certain homological assumptions) hyperbolic groups obtained as graphs of free groups amalgamated over cyclic subgroups contain surface subgroups. 
Kim and Oum \cite{kim-oum} showed the same is true if $G$ is a one-ended double of a free group (this work was preceded by the paper of Gordon and Wilton \cite{gordon-wilton}).

The purpose of this paper is to show that if $G$ is a hyperbolic group with $\pt{G} \approx \Sp^2$, and if $G$ contains ``enough"  quasi-convex surface subgroups, then $G$ is isomorphic to the fundamental group of closed hyperbolic 3-manifold. In the following definition we make precise what we mean by $G$ containing  ``enough" quasi-convex surface subgroups.

\begin{definition} Let $G$ be a hyperbolic group. We say that two (distinct) points $p,q \in \pt{G}$ are separated by a quasi-convex, co-dimension 1 subgroup $H<G$ if they lie in different connected components of the set $\pt{G} \setminus \pt{H}$ (sometimes we also say that $p$ and $q$ are separated by the limit set of $H$).
We say that $G$ contains enough co-dimension 1, quasi-convex subgroups if every two  points $p,q \in \pt{G}$ can be separated by a co-dimension 1, quasi-convex subgroup $H<G$.
\end{definition}

When $\pt{G}\approx \Sp^2$, quasi-convex surface subgroups are examples of co-dimension 1 subgroups of $G$. 
If every two points in $\pt{G}$ can be separated by a quasi-convex surface subgroup we say that $G$ contains enough quasi-convex surface subgroups. The main result of this paper is the following criterion for the Cannon's Conjecture:

\begin{theorem}[Criterion for Cannon's  Conjecture]\label{thm-main} Let $G$ be a hyperbolic group (that acts effectively on its boundary) with $\pt{G}\approx \Sp^2$. Suppose that $G$ contains enough quasi-convex surface subgroups. Then $G$ is isomorphic to a Kleinian group.
\end{theorem}

Kahn and Markovic \cite{kahn-markovic}  showed that the fundamental group of a closed hyperbolic 3-manifold contains enough quasi-convex surface subgroups and thus by the above theorem we have the following equivalence stated in the abstract above: 

\begin{alltt}
\normalfont
\emph{A hyperbolic group that acts effectively on its boundary, and whose 
boundary is homeomorphic to the 2-sphere, is isomorphic to a Kleinian 
group if and only if it contains enough quasi-convex surface subgroups.} 
\end{alltt}

\subsection{Cubulation and separability in hyperbolic groups} The sizable part of the argument behind the proof of the Criterion for Cannon's conjecture relies on the following theorem recently proved by Agol (see Theorem 1.1 and Corollary 1.2  in \cite{agol}):

\begin{theorem}[Agol]\label{thm-A} Let $G$ be a cubulated hyperbolic group (a group is cubulated if it is acting properly and co-compactly on a CAT(0) cube complex). Then  quasi-convex subgroups of $G$ are separable. 
\end{theorem}

\begin{remark} This theorem is stated as Corollary 1.2 in \cite{agol}. If a hyperbolic groups $G$ is cubulated then it is proved that  $G$ virtually embeds into a Right Angled Artin Group and thus $G$ has a finite index subgroup whose quasi-convex subgroups are separable.  
On the other hand, the notion of quasi-convex subgroups being separable is invariant under passing onto subgroups or onto finite index supergroups (see  Lemma 2.2.2. in \cite{l-r}) and we conclude that quasi-convex subgroups of $G$ are separable.
\end{remark}
\vskip .3cm
Recall that a subgroup $H<G$ is separable if there is a sequence $G_n<G$ of finite index subgroups of $G$ such that 
$$
\bigcap G_n=H.
$$
Also, for a subgroup $H<G$ and $g \in G$ we will use the standard abbreviation $H^g=g^{-1}Hg$. Another important ingredient we need is the following theorem of Bergeron-Wise (see the statement and the proof of Theorem 1.4 in \cite{bergeron-wise}) that builds on the work of Sageev \cite{sageev}:

\begin{theorem}[Bergeron-Wise]\label{thm-BW} Let G be a hyperbolic group that contains enough  quasi-convex, co-dimension 1 subgroups. Then 
\begin{enumerate}
\item One can find finitely many co-dimension 1, quasi-convex subgroups $H_1,...,H_n<G$, such that for every two distinct points 
$p,q \in \pt{G}$, there exists $1\le i \le n$ and $g \in G$ such that $p$ and $q$ are separated by  $H^{g}_i$.
\item  The group $G$ is cubulated. 
\end{enumerate}
\end{theorem}

In fact, Bergeron and Wise show that the group $G$ acts on the CAT(0) cube complex associated to the subgroups $H_1,...,H_n$ (the associated CAT(0) cube complex was introduced by Sageev in \cite{sageev}). We also say that $G$ is cubulated by the groups $H_1,..,H_n$.
In order to apply Theorem \ref{thm-A} and conclude that quasi-convex subgroups of $G$ are separable we only need to know that the property $(2)$ holds in the previous theorem (that is, we need to know that $G$ is cubulated).
However, we will make other important use of the property $(1)$ beside its primary role in showing that $G$ is cubulated. This will be explained at the beginning of Section 3 below.

We record the main corollary of Theorem \ref{thm-A} and the part $(2)$ of Theorem \ref{thm-BW}:

\begin{corollary}\label{cor-main-1}  Let $G$ be a hyperbolic group that contains enough quasi-convex, co-dimension 1 subgroups. Then every quasi-convex subgroup of $G$ is separable.
(In particular, if $G$ is a hyperbolic group with $\pt{G}\approx \Sp^2$ that  contains enough quasi-convex surface subgroups, then every quasi-convex subgroup of $G$ is separable.)
\end{corollary}

As we already mentioned, fundamental groups of hyperbolic 3-manifolds are cubulated \cite{kahn-markovic}, \cite{bergeron-wise}. Thus the main application of this corollary in 3-dimensional topology is that it yields a proof of the Virtual Haken Conjecture \cite{agol}.  

In this paper we show that this result may have other important applications. To illustrate this point we show how Corollary \ref{cor-main-1} together with the work of Tukia \cite{tukia} gives a new proof of the following theorem of Gabai, Casson-Jungreis.

\begin{theorem}[Gabai, Casson-Jungreis]\label{thm-GCJ} Let G be a word-hyperbolic group that acts effectively on its boundary and such that $\pt{G}$ is homeomorphic to $\Sp^1$. Then $G$ is Fuchsian.
\end{theorem}

\begin{remark}  In fact, Gabai and Casson-Jungreis prove a stronger result. The only assumption they need is that $G$ is a convergence group acting on $\Sp^1$ (such $G$ can contain ``parabolic" elements and thus it does not have to be a  hyperbolic group, and also the limit set of $G$ need not be the entire circle). But the previous theorem is a major case in their study and itself is a very strong result.
\end{remark}

\begin{proof} The group $G$ is hyperbolic and $\partial{G} \approx \Sp^1$. Thus infinite order cyclic subgroup of $G$ are co-dimension 1 quasi-convex subgroups. Moreover, every two pints on $\Sp^1$ can be separated by the limit set of a such  cyclic group (the limit set of such groups consists of two points on $\Sp^1$). It follows from Corollary \ref{cor-main-1} that $G$ is residually finite and since $G$ is hyperbolic there exists a finite index subgroup $\wt{G}<G$ that is torsion-free. Then by 
Corollary 3B and Theorem 6B from \cite{tukia} it follows that $G$ is topologically conjugated to a Fuchsian group. 
\end{proof}

\subsection{Proof of the Criterion for Cannon's Conjecture}  Suppose that $G$ is a hyperbolic group, with $\pt{G}\approx \Sp^2$,  and suppose that $G$ contains enough quasi-convex surface subgroups.
The main step in the proof is to show that $G$ is virtually isomorphic to a fundamental group of a closed 3-manifold. We prove the following theorem in the remainder of this paper:

\begin{theorem}\label{thm-1} Let $G$ be a hyperbolic group with $\pt{G}\approx \Sp^2$ and suppose that $G$ contains enough quasi-convex surface subgroups. Then $G$ contains a torsion-free finite index subgroup $\wt{G}<G$ that is isomorphic to the fundamental group of an oriented 3-manifold $M$ that contains a closed incompressible surface.
\end{theorem}

\begin{remark} Assume that $G$ is a hyperbolic group with $\pt{G}\approx \Sp^n$, and suppose that $G$ contains enough quasi-convex, $(n-1)$-negatively curved manifold subgroups (a group $H$ is $(n-1)$-negatively curved manifold group if $H$ is isomorphic to the fundamental group of a closed $(n-1)$-dimensional Riemannian manifold of negative curvature). Using similar methods one may be able to show that $G$ is virtually  isomorphic to the fundamental group of a $n$-dimensional topological manifold (whose universal cover is homeomorphic to the $n$-dimensional ball).
\end{remark}

Assuming Theorem \ref{thm-1} the rest of the proof of Theorem \ref{thm-main} runs as follows. We first show that the Cayley graph of $\wt{G}$ is quasi-isometric to the hyperbolic space $\D^3$.

Let $M$ be the manifold from Theorem \ref{thm-1}, that is the oriented 3-manifold such that $\wt{G} \cong \pi_1(M)$. We do not if $M$ is compact but we do know that $\pi_1(M)$ is finitely generated since it is a hyperbolic group. 
Applying the Scott's Core Theorem \cite{scott} we find that there exists a compact 3-manifold $\wt{M}$ whose fundamental group is isomorphic to $\wt{G}$ (the new manifold $\wt{M}$ is a ``compact core" of $M$ and it is embedded in $M$). 
In particular, the new manifold $\wt{M}$ is Haken (because $M$ is Haken). For simplicity we set $M=\wt{M}$, that is $M$ is a compact 3-manifold and $\pi_1(M) \cong \wt{G}$.

We may also assume that no component of the boundary of $M$ is $\Sp^2$ (otherwise we just glue in a 3-disc and this does not change the fundamental group of $M$). We now show that $M$ is a  closed hyperbolic manifold.

By the Thurston Hyperbolization Theorem for Haken 3-manifolds (see Theorem 1.43 in \cite{kapovich}) the manifold $M$ admits a  complete geometrically finite hyperbolic structure. That is, there exists a geometrically finite Kleinian group $\Gamma$ acting on $\D^3$ with the following properties:

\begin{enumerate}
\item $\pi_1(M) \cong \Gamma$.  
\item  There exists $\epsilon>0$ such that $M$ is homeomorphic to the $\epsilon$-neighborhood of the quotient $C(\Gamma)/\Gamma$,  where  $C(\Gamma) \subset \D^3$ is the convex core of the limit set $\Lambda(\Gamma) \subset \Sp^2$ of $\Gamma$. 
\end{enumerate}

In particular, $\Gamma$ acts co-compactly and properly on the $\epsilon$-neighborhood of $C(\Gamma)$. Thus, it follows from the Milnor-Svarc Lemma that the Cayley graph of $\Gamma$ is quasi-isometric to $C(\Gamma)$ and in particular the boundaries at infinity of the two spaces are homeomorphic. Since $\Lambda(\Gamma)$ is the boundary at infinity of the $\epsilon$-neighborhood of $C(\Gamma)$ we conclude that $\Lambda(\Gamma)=\Sp^2$.  
Since $\Gamma$ is geometrically finite we conclude that $C(\Gamma)=\D^3$ and the Cayley graph of $\wt{G}$ is quasi-isometric to $\D^3$. Therefore $M$ is homeomorphic to the closed hyperbolic 3-manifold $\D^3/\Gamma$. 

It remains to show that $G$ is isomorphic to a Kleinian group. We saw that the Cayley graph of $\wt{G}$ is quasi-isometric to $\D^3$. Since $\wt{G}$ is a finite index subgroup of $G$ it follows that the Cayley graph of $\wt{G}$ is  quasi-isometric to the Cayley graph of $G$. Thus,  the Cayley graph of $G$ is quasi-isometric to $\D^3$. Then it follows from Theorem 9.2 in \cite{k-b} that $G$ is isomorphic to a Kleinian group (this last statement follows from the  Sullivan-Tukia theorem which states that quasi-conformal groups of homeomorphisms of $\Sp^2$ are conjugated to Kleinian groups).

\subsection{Outline of the paper} It remains to prove Theorem \ref{thm-1}. That is, we need to show that a hyperbolic group $G$ with $\pt{G}\approx \Sp^2$ and  that contains enough quasi-convex surface subgroups has a finite index subgroup that is isomorphic to the fundamental group of a Haken 3-manifold.

This is done in two steps. The first step is done in the next section where we prove Theorem  \ref{thm-X-1}. This theorem is a (virtually)  strengthened version of Theorem \ref{thm-BW}. Namely, we show that any hyperbolic group $G$ that is cubulated by quasi-convex, co-dimension 1 subgroups, contains a finite index and torsion-free subgroup  $\wt{G}<G$ which can be cubulated by malnormal, quasi-convex, co-dimension 1 subgroups. Moreover  if $\pt{G}\approx \Sp^2$ and $G$ can be cubulated by quasi-convex surface subgroups then $\wt{G}$  can be cubulated by  malnormal, quasi-convex, surface subgroups.

The second step is to show that the group $\wt{G}$ is isomorphic to the fundamental group of a 3-manifold that contains a closed incompressible surface. We do this by showing that the action of $\wt{G}$ on $\Sp^2$ can be extended to a free and properly discontinuous action by homeomorphisms on $\D^3$ (in fact we will show a stronger property that $\wt{G}$ acts on $\D^3$ freely as a convergence group). Thus $M \approx \D^3/\wt{G}$ is a 3-manifold whose fundamental group is isomorphic to $\wt{G}$ ($M$ contains an incompressible surface since $\wt{G}$ contains  malnormal quasi-convex surface groups).

\begin{remark} We use a considerable firepower in the proof of the Criterion for Cannon's Conjecture. Besides  Agol's Theorem \ref{thm-A} and Bergeron-Wise Theorem \ref{thm-BW} we also make essential use of the Thurston's Haken Hyperbolization Theorem, theory of quasi-convex subgroups of hyperbolic groups (mainly results from the paper \cite{g-m-r-s} by Gitik-Mitra-Rips-Sageev), and we develop the short theory of $G$-complexes (this name is inspired by Tukia's notation in \cite{tukia} but our construction and theory of $G$-complexes is quite different than his theory from \cite{tukia}).

However, we do not rely on Perelman's Geometrization Theorem. In light of the Surface Subgroup Theorem \cite{kahn-markovic}, there is a realistic hope that one may be able to prove the existence of enough surface subgroups in $G$ which would then yield a a proof of the Cannon's Conjecture which together with the result of Gabai-Meyerhoff-Thurston \cite{g-m-t} would give a new proof of the major chunk of the Hyperbolization Theorem for 3-manifolds. 
\end{remark}

\subsection{Acknowledgement} Ian Agol suggested independently that Theorem \ref{thm-main}  should be true. In fact, I am grateful to Ian for  reading this manuscript and sending me detailed comments that have improved the paper. 
Also, I would very much like to thank Leonid Potyagailo and Victor Gerasimov for sending me suggestions and corrections.

\section{Virtual cubulation by malnormal subgroups} 
In this section  we show that if a hyperbolic group $G$ is cubulated by quasi-convex, co-dimension 1 subgroups, then it contains a torsion-free,  finite index subgroup $\wt{G}$  which is cubulated by malnormal, quasi-convex, co-dimension 1 subgroups. In particular, if $\pt{G}\approx \Sp^2$, and $G$ is cubulated by quasi-convex surface subgroups,  we show that $\wt{G}$ is cubulated by malnormal, quasi-convex surface subgroups. Then in the next section we show that this subgroup $\wt{G}$ is isomorphic to the fundamental group of a 3-manifold with a closed incompressible 
(meaning essential and embedded) subsurface.

\subsection{Preliminary lemmas}  The main result of this subsection is Lemma \ref{lemma-X-0} which may be of independent interest.  We refer the reader to the paper \cite{g-m-r-s} by Gitik-Mitra-Rips-Sageev as we will use various supporting lemmas from that paper throughout.
  
The following lemma follows immediately from Lemma 1.2 in \cite{g-m-r-s} and it was stated in \cite{a-g-m} as Lemma 3.4. 

\begin{lemma}\label{lemma-AGM} Let $G$ be a hyperbolic group and  $H<G$ a  quasi-convex subgroup. Then the group $H$ contains only finitely many $H$-conjugacy classes of subgroups $H \cap H^g$, $g \in G$, such that $|H \cap H^g|=\infty$.
\end{lemma}

Recall that the boundary $\pt{G}$ of a hyperbolic group $G$ can be endowed with a visual metric such that $\pt{G}$ is a compact metric space. Once and for all we fix such a metric on $\pt{G}$. Then the group $G$ acts on $\pt{G}$ by homeomorphisms as a convergence group with respect to the topology on $\pt{G}$. We use the same letter $g$ to denote a group element and the corresponding homeomorphism of $\pt{G}$.

Given a set $S \subset G$, we let $\Lambda(S)$ denote the set of accumulation points on $\pt{G}$ of the set $S$ (we refer to $\Lambda(S)$ as the limit set of $S$). If $H<G$ is a subgroup and $g \in G$, then $\Lambda(H^g)=g^{-1}(\Lambda(H))$.
The following lemma is Corollary 2.5 from \cite{g-m-r-s}:

\begin{lemma}\label{lemma-GMRS} Let $H<G$ be a quasi-convex subgroup and let $\delta>0$. Then there exists a constant $R=R(\delta,H)$ such that there is at most $R$  different groups of the form $H^g$, $g \in G$, for which $\diam(\Lambda(H^g))>\delta$  
(here $\diam$ denotes the diameter of a set in $\pt{G}$ with respect to its visual metric).
\end{lemma}

Recall that a subgroup $F<G$ is malnormal if $F\cap F^g=\id$, for any $g \in G \setminus F$. We say that $F<G$ is almost malnormal if  for any $g \in G \setminus F$, the group  $F\cap F^g$ is finite. For a closed subset $P\subset \pt{G}$, by $\stab(P)<G$ we denote the subgroup of $G$ that leaves $P$ invariant as a set (if we want to emphasize that this is with respect to the group $G$ we write $\stab_G(P)$ instead). We say that a quasi-convex subgroup $F<G$ is a maximal subgroup in $G$ if $\stab{\Lambda(F)}=F$.

\begin{remark} If $F<G$ is quasi-convex, almost malnormal and infinite subgroup of $G$  then  $F$ is automatically maximal.
\end{remark}

The next lemma is the main result of this subsection.

\begin{lemma}\label{lemma-X-0} Let $G$ be a hyperbolic group and $H<G$ a separable, quasi-convex subgroup. Then there exists a finite index subgroup $\wh{G}<G$ that contains $H$ and such that $H$ is almost malnormal in $\wh{G}$.
\end{lemma}

\begin{proof}  Without loss of generality we may assume that $H$ is infinite (otherwise $H$ is almost malnormal in $G$). Since $H$ is quasi-convex, it follows from Lemma 2.9 in \cite{g-m-r-s} that 
$H$ is a finite index subgroup of $\stab_G{\Lambda(H)}$. Since $H$ is separable we can find a finite index subgroup of $G$ which contains $H$ as a maximal subgroup. So without loss of generality we may assume that $H$ is a maximal, separable, quasi-convex subgroup of $G$.

By Lemma \ref{lemma-AGM} there exists finitely many non-trivial subgroups 
$$
F_1,...,F_k<H,
$$ 
with the following property. Suppose that $g\in G \setminus H$ is  such that the intersection $H^g \cap H$ is non-trivial 
and $|H^g \cap H|=\infty$. Then there exists $h\in H$ so that 
$$
H^{gh} \cap H=F_i, 
$$
for some $1\le i \le k$. Each $F_i$ is an infinite and quasi-convex group ($F_i$ is quasi-convex as an intersection of two quasi-convex subgroups). Thus $\diam(\Lambda(F_i))>0$.  Therefore, there exists $\delta_0>0$ so that $\diam(\Lambda(F_i))>\delta_0$, for each $i$. 

Let $g \in G \setminus H$ be such that $H^g \cap H=F_i$, for some $i$. Then $\diam(\Lambda(H^g)) \ge \diam(\Lambda(F_i)) >\delta_0$. Thus, it follows from Lemma \ref{lemma-GMRS} that there are at most  $R_0=R(\delta_0)$ 
(where $R$ is the constant from Lemma \ref{lemma-GMRS})  different groups of the form $H^g$, $g \in G \setminus H$, such that $H^g \cap H=F_i$, for some $i$. Therefore, there exists elements $g_1,...,g_{R_{0}} \in G \setminus H$ such that if $H^g \cap H=F_i$, for some $g \in G$ and some $i$, then $H^g=H^{g_{j}}$ for some $1\le j \le R_0$.

Since $H$ is separable there exists a finite index subgroup $\wh{G}<G$ that contains $H$ and that does not contain any of the elements $g_1,...,g_{R_{0}}$ defined above. 
We show that $H$ is almost malnormal in $\wh{G}$. Suppose on the contrary that for some $g \in \wh{G}\setminus H$ we have $H^g \cap H$ is a non-trivial subgroup of $H$ and $|H^g \cap H|=\infty$. As we saw, 
there exists $h \in H$ such that $H^{gh} \cap H=F_i$, for some $i$. But this implies that $H^{gh}=H^{g_{j}}$ for some $j$. Thus $H^{ghg^{-1}_{j}}=H$ and therefore $ghg^{-1}_{j} \in \stab_G{\Lambda(H)}$. Recall that we assume that $H$ is maximal. It follows that  $ghg^{-1}_{j} \in H$. However, since both $g$ and $h$ belong to $\wh{G}$, and $H<\wh{G}$, we conclude that $g_{j} \in \wh{G}$, which is a contradiction. This completes the proof.

\end{proof}

\subsection{Virtually cubulating $G$ by malnormal subgroups} The following theorem is the main result of this section.

\begin{theorem}\label{thm-X-1} Let $G$ be a hyperbolic group such that every two  points in $\pt{G}$ can be separated by a  co-dimension 1, quasi-convex subgroup. 
Then there exists a torsion-free, finite index subgroup $\wt{G}<G$, and malnormal, co-dimension 1, quasi-convex subgroups $H_1,...,H_m<\wt{G}$ such that every two distinct point of $\pt{\wt{G}}$ can be separated by $H^{g}_i$, for some $g \in \wt{G}$ and $1\le i\le m$. 
\end{theorem}

\begin{proof} It follows from Theorem \ref{thm-A} that $G$ is residually finite, so by passing onto a finite index subgroup if necessary, we may assume that $G$ is torsion-free.

By Theorem \ref{thm-BW}, we know that there exist  co-dimension 1, quasi-convex subgroup $F_1,...,F_m<G$ such that every two points of $\pt{G}$ can be separated by some conjugate $F^{g}_{i}$. 
Let $\wh{G}_{F_{i}}<G$ denote a finite index subgroup of $G$ that contains $F_i$ as a  malnormal  subgroup (by Lemma \ref{lemma-X-0} there is a finite index subgroup $\wh{G}_{F_{i}}<G$ that contains $F_i$ as an almost malnormal group, but since we assumed that $G$ is torsion free then $F_i$ it follows that $F_i$ is malnormal in  $\wh{G}_{F_{i}}$). If $g \in G$, we observe that the group $g^{-1}\wh{G}_{F_{i}}g=\wh{G}^{g}_{F_{i}}$ is a finite index subgroup of $G$ (of the same index as $\wh{G}_{F_{i}}$) that contains $F_{i}^g$ as a malnormal subgroup. 

Let $\mathcal{F}$ denote the collection of all groups $F^{g}_{i}$, for all $g \in G$ and $1\le i \le n$ (thus $\mathcal{F}$ is an infinite collection of groups but it contains only finitely many conjugacy classes). 
For each group $F \in \mathcal{F}$, let $\wh{G}_F<G$ denote the corresponding finite index subgroup of $G$ that contains $F$ as malnormal subgroup (that is, if $F=F^{g}_i$ then $\wh{G}_F=\wh{G}^{g}_{F_{i}}$). 
By construction, there are only finitely many conjugacy classes of groups $\wh{G}_F$, $F \in \mathcal{F}$, so there exists $N>0$ such that each group $\wh{G}_F$ has index at most $N$ in $G$. Since $G$ is hyperbolic (and thus finitely generated) there are at most finitely many different subgroups of $G$ of index at most $N$ (see $4.$ on page 128. in \cite{hall}) .  

Set 
$$
\wt{G}=\bigcap_{F \in \mathcal{F}} \wh{G}_F.
$$
Then $\wt{G}$ is a finite index subgroup of $G$. Set $H_F=F \cap \wt{G}$, for $F \in \mathcal{F}$. Then $H_F$ is a finite index subgroup of $F$,  so it follows that  $H_F<\wt{G}$ is quasi-convex and co-dimension 1 subgroup.  
Also, $\Lambda(H_F)=\Lambda(F)$, and thus  every two points in $\pt{G}$ are separated by some $H_F$ (for some $F \in \mathcal{F}$). 

We now show that each $H_F$ is a malnormal subgroup of $\wt{G}$. Let $g \in \wt{G} \setminus H_F$, and suppose that  $|H^{g}_{F} \cap H_F|=\infty$. Since $(H^{g}_{F} \cap H_F) \subset (F^{g} \cap F)$, it follows that $|F^{g} \cap F|=\infty$.  But $g \in \wt{G}<\wh{G}_F$, and since $F$ is malnormal (and thus maximal) in $\wh{G}_F$ it follows that $g \in F$. Thus, $g$ belongs to $F$ and also $g \in \wt{G}$. It follows  $g \in F \cap \wt{G}=H_F$, which is a contradiction, and we have proved that $H_F$ is  malnormal in $\wt{G}$.

On the other hand,  there are finitely many $G$-conjugacy classes of groups $F \in \mathcal{F}$, so we conclude that there are only finitely many $\wt{G}$-conjugacy classes of groups $H_F$, $F \in \mathcal{F}$. Let $H_1,...,H_m<\wt{G}$ be representatives of these finitely many conjugacy classes. Then each $H_j<\wt{G}$ is malnormal, co-dimension 1, quasi-convex subgroup, and every two  points of $\pt{\wt{G}}$ can be separated by $H^{g}_i$, for some $g \in \wt{G}$ and $1\le i\le m$. 
\end{proof}
In the special case when the boundary of $G$ is homeomorphic to $\Sp^2$ and if every two points of $\pt{G}$ can be separated by a quasi-convex, surface subgroup, the above theorem can be restated as follows:

\begin{theorem}\label{thm-X-2} Let $G$ be a hyperbolic group with $\pt{G}\approx \Sp^2$ and such that every two  points in $\pt{G}$ can be separated by a quasi-convex, surface subgroup. 
Then there exists a finite index, torsion-free subgroup $\wt{G}<G$, and malnormal,  quasi-convex, surface groups $H_1,...,H_m<\wt{G}$ such that every two distinct point of $\pt{\wt{G}}$ can be separated by $H^{g}_i$, for some $g \in \wt{G}$ and $1\le i\le m$. Moreover, elements of $\wt{G}$ act as orientation preserving homeomorphisms on $\pt{G}\approx \Sp^{2}$ (with respect to the standard orientation on $\Sp^2$).
\end{theorem}
\begin{proof} If $G$ contains an element that acts as  an orientation reversing homeomorphism on $\Sp^2$, then $G$ contains an index two subgroup whose elements act on $\pt{G}\approx \Sp^2$ as orientation preserving homeomorphisms. In this case, we replace $G$ by this index two subgroup (which we also call $G$).

By the previous theorem we can find a torsion-free, finite index subgroup $\wt{G}<G$ that contains  malnormal, quasi-convex, surface subgroups $H_1,...,H_m<\wt{G}$ such that every two distinct points of $\pt{\wt{G}}$ can be separated by $H^{g}_i$, for some $g \in \wt{G}$ and $1\le i\le m$. Denote by $\mathcal{H}$ the collection of surface subgroups  $H^{g}_i$, where $g \in \wt{G}$ and $1\le i\le m$. Then every two points in $\pt{G}$ can be separated by a group  from $\mathcal{H}$.

\end{proof}

\section{Extending the action of $G$ from $\Sp^2$ to the 3-ball $\D^3$ and the proof of Theorem \ref{thm-1} } It remains to prove Theorem \ref{thm-1}. The proof follows immediately from  Theorem \ref{thm-X-1} and the following theorem that we prove in the remainder of this section.

\begin{theorem}\label{thm-1-new} Let $G$ be a torsion-free hyperbolic group with $\pt{G}\approx \Sp^2$, which acts  by orientation preserving homeomorphisms on $\pt{G}$. Suppose that $H_i<G$, $1\le i\le k$, are quasi-convex, malnormal surface subgroups, such that every two points of $\pt{G}$ are separated by the limit set of some $G$-conjugate of some $H_i$. Then $G$ is isomorphic to the fundamental group of a 3-manifold. 
\end{theorem}

We briefly outline an idea of the proof. The group $G$ acts as a convergence group on $\Sp^2=\pt{\D}^3$. For each surface subgroup $H_i<G$ we find a disjoint collection of embedded 2-discs in $\D^3$ that are bounded by the limit sets of $G$-conjugates of $H_i$. We show that the action of $G$ can be extended from $\Sp^2$ to a convergence action on this new subset of $\D^3$. 

This is an example of a $G$-complex which by definition consists of a closed subset $\Kr\subset \overline{\D}^3$ (such that $\Sp^2 \subset \Kr$ and that each component of $\D^3 \setminus \Kr$ is a topological disc) and a $G$-action on $\Kr$ (in the above example the closed set $\Kr$ is the union of the embedded 2-discs and $\Sp^2$). 
We develop a short theory of $G$-complexes and define an ``intersection" of two $G$-complexes. This is defined in an intelligent way and not in the obvious (and inadequate) way by taking the union of the corresponding sets  $\Kr$'s.  We also show that if the stabilizer of each connected component of $\overline{\D}^3 \setminus \Kr$ is trivial then the $G$-action extends to $\overline{D}^3$.

We finish the argument by taking the ``intersection" of the $G$-complexes associated with $H_i$'s. This new complex has the property that the stabilizer of each connected component of $\D^3\setminus \Kr$ is trivial, and this property follows from the assumption that every two points of $\pt{G}$ can be separated by the limit set of a conjugate of some $H_i$. Therefore, we can extend the action of $G$ to $\D^3$ (we show that this action is free and convergence) and the required 3-manifold is constructed as the quotient of $\D^3$ by this $G$-action.

\subsection{Notation and basic definitions} Let $X$ be a compact metric space and let $F$ be a subgroup of the group $\Homeo(X)$ of homeomorphisms of $X$. We say that $F$ is a convergence group
if for every sequence of different $f_n \in F$, there exists a subsequence $f_{n_{k}}$ and points $a,b \in X$ such that $f_{n_{k}} \to a$ uniformly on compact subsets of $X\setminus \{b\}$. 
Note that in this case the sequence of inverse maps  $f^{-1}_{n_{k}}$ converges to $b$ uniformly on compact subsets of $X\setminus \{a\}$. 

Let $G$ denote a group. A $G$-action on metric space $X$ is a monomorphism $\mu:G \to \Homeo(X)$ (we assume that $G$-actions are effective, that is,  $\mu$ is a monomorphism). 
As usual, we use the convention  $\mu(fg)=\mu(g)\circ \mu(f)$, $f,g \in G$.  We say that a $G$-action on $X$ is a convergence $G$-action if the corresponding group of homeomorphisms $\mu(G)<\Homeo(X)$ is a convergence group.

In what follows we identify $\Sp^2$ with the boundary of the open unit 3-ball $\D^3$, that is $\Sp^2=\pt{\D}^3$. The closed 3-ball is $\overline{\D}^3$. We consider $\overline{\D}^3$ (and all of its subsets) as the metric space with the metric it inherits from $\R^3$ (here $\R^3$ is equipped with the standard Euclidean metric). We write $d(x,y)$ and $\diam(S)$  to denote respectively the distance between points $x$ and $y$  and the diameter of a set $S \subset \overline{\D}^3$  with respect to this metric on $\overline{\D}^3$ (or its subsets).

\subsection{G-Complex}  We start by giving a definition of a generalized cell decomposition in $\overline{\D}^{3}$ (we do not use or rely on any of the standard theory about cell complexes).

\begin{definition} A pair $(\Kr,\Ur)$ is called a generalized cell decomposition of $\D^3$ if the following holds:

\begin{enumerate} 

\item $\Kr $ is a closed subset of $\overline{\D}^3$ and $\Kr$ contains $\Sp^2=\pt{\D}^3$.

\item By  $\Ur$ we denote the collection of connected components of $\D^3 \setminus \Kr$. We require that every open set $U \in \Ur$ is homeomorphic to $\D^3$ and that its boundary $\pt{U}$ is homeomorphic 
to $\Sp^2$.

\item For any $\delta>0$, there exists $N=N(\delta)$ such that there are at most $N$ components in $\Ur$ whose diameter is greater than $\delta$.

\end{enumerate}
\end{definition}

There are two obvious examples of generalized cell decompositions. The first one is when the closed set $\Kr$ is as small as possible, that is $\Kr=\Sp^2$. Then $\Ur$ contains $\D^3$ and this is the only set in $\Ur$. 
The second example is when $\Kr$ is as large as possible, that is $\Kr=\overline{\D}^3$. Then $\Ur$ is the empty collection. 

We say that a generalized cell complex $(\Kr,\Ur)$ is a refinement of a generalized cell complex $(\Kr',\Ur')$ if 
$\Kr'\subset \Kr$ (clearly this induces a partial order on the collection of all generalized cell complexes and two examples we mentioned are respectively the maximal and minimal element  with respect to this partial order).

Let $G$ denote a group.  A $G$-action on a generalized cell complex $(\Kr,\Ur)$  is a $G$-action $\mu:G \to \Homeo(\Kr)$ such that:
\begin{enumerate}
\item  $\mu(g)(\Sp^2)=\Sp^2$ for each $g \in G$.
\item For each $g \in G$ and every $U \in \Ur$, there exists $U' \in \Ur$ such that $\mu(g)(\pt{U})=\pt{U'}$.
\end{enumerate}

Thus, every $G$-action on $(\Kr,\Ur)$ induces a $G$-action on $\Sp^2$. We will only consider $G$-actions such that the induced action on $\Sp^2$ is by orientation preserving homeomorphisms. 

Fix a $G$-action $\mu:G \to \Homeo(\Kr)$ on a generalized cell complex $(\Kr,\Ur)$. We say that this $G$-action $\mu$ is free if for every $g \in G$, $g \ne \id$, the homeomorphism $\mu(g)$ has no fixed points in $\Kr \cap \D^3$.
Fix $U \in \Ur$. We let $g \in \stab_{\mu}(U)$ if $\mu(g)(\pt{U})=\pt{U}$. The stabilizer $\stab_{\mu}(U)$ is a subgroup of $G$.

Next, we define the notion of a $G$-complex.

\begin{definition}\label{def-rev} Let $G$ denote a group and let $(\Kr,\Ur)$  be a generalized cell complex in $\D^3$. We say that the triple $(\mu,\Kr,\Ur)$ is a $G$-complex if $\mu:G \to \Homeo(\Kr)$ is a $G$-action on $(\Kr,\Ur)$  with the following property:
\begin{enumerate} 

\item  For every $U \in \Ur$ there  exists a homeomorphism $\phi^U:(\overline{\D}^3,\Sp^2) \to (\overline{U},\pt{U})$, between the corresponding pairs, 
such that the homeomorphism  $\phi^U:\Sp^2 \to \pt{U}$ fixes every point in $\pt{U}\cap \Sp^2$, and $\phi^U$ is equivariant with respect to the group $\mu(G)$, that is:
$$
\mu(g)\circ \phi^U=\phi^{\mu(g)(U)} \circ \mu(g),\, \text{on}\,\, \Sp^2, \,\, \text{for every}\,\,  g \in G.
$$
\end{enumerate}

If the $G$-action $\mu$ is free we say that $(\mu,\Kr,\Ur)$ is a free $G$-complex.

\end{definition}

Given a $G$-complex $(\mu,\Kr,\Ur)$, for $U \in \Ur$ there may exist more than one choice for the corresponding homeomorphism $\phi^U:(\overline{\D}^3,\Sp^2) \to (\overline{U},\pt{U})$. If we want to remember  choices of homeomorphisms $\phi^U$, $U \in \Ur$,  then we  say that the 4-tuple $(\mu,\Kr,\Ur,\phi^{\Ur})$ is a {\it marked $G$-complex}.

Observe that $g \in G$ induces a permutation on $\Ur$ by $\mu(g)(U)=V$ if $\mu(g)(\pt{U})=\pt{V}$ (it follows from the definition of a generalized cell complex that if $\pt{V}_1=\pt{V}_2$ for some $V_1,V_2 \in \Ur$ then $V_1=V_2$).

\begin{definition} We say that a $G$-complex $(\mu,\Kr,\Ur)$ is a convergence $G$-complex if the group $\mu(G)$ acts on  $\Kr$ as  a convergence group. 
\end{definition}

We observe that a convergence $G$-complex is free providing that $G$ has no torsion. This is seen as follows. 
Let $(\mu,\Kr,\Ur)$ be a convergence $G$-complex. Then for each infinite sequence of different elements $g_n \in G$ there are points $a,b \in \Kr$ such that $\mu(g_n) \to a$ uniformly on compact subsets of $\Kr \setminus \{b\}$. Since $\Sp^2 \subset \Kr$ and $\mu(g)(\Sp^2)=\Sp^2$ for each $g \in G$, it follows that $a,b \in \Sp^2$. If the group $G$ has no torsion it follows that for $g \in G \setminus \{\id \}$, the homeomorphism $\mu(g)$ has no fixed points in $K \cap \D^3$.

\begin{proposition}\label{prop-lako} Let $\Kr_0=\overline{\D}^3$ and let $(\mu,\Kr_0,\Ur_0)$ be a convergence and free $G$-complex (here $\Ur_0$ denotes an empty collection). 
Then the quotient  $M \approx \D^{3}/ \mu(G)$ is a $3$-manifold whose  fundamental group $\pi_1(M)$ is isomorphic to $G$.
\end{proposition}

\begin{proof} The homeomorphism $\mu(g)$  (for a non-trivial elements $g \in G$) has no fixed points in $\D^3$, thus the $G$-action on $\D^3$ is free.  
Since the $G$-complex is convergence, the group $\mu(G)$ is a convergence group on $\overline{\D}^3$ and therefore the $G$-action is properly discontinuous in $\D^3$. So, the quotient  $M \approx \D^{3}/ \mu(G)$ is a $3$-manifold. Since $\D^3$ is contractible we have $\pi_1(M) \approx G$.
\end{proof}

\subsection{Refinement of a $G$-complex}  We say that a $G$-complex $(\mu,\Kr,\Ur)$ is a refinement of a $G$-complex $(\mu',\Kr',\Ur')$ if $\Kr' \subset \Kr$ and $\mu(g)=\mu'(g)$ on $\Kr'$ (one can show that this induces a partial order on the collection of all $G$-complexes). Then the generalized cell complex $(\Kr,\Ur)$ is a refinement of $(\Kr',\Ur')$. Moreover, if $U \subset U'$ for some $U \in \Ur$ and $U'\in \Ur'$, then $\stab_{\mu}(U)<\stab_{\mu'}(U')$.

The following lemma is the main tool we will use to prove that  $G$-complexes are convergence.

\begin{lemma}\label{lemma-glava} Suppose that  a $G$-complex $(\mu,\Kr,\Ur)$ is a refinement of a convergence $G$-complex $(\mu',\Kr',\Ur')$. If for every $U' \in \Ur'$ the group $\mu(\stab_{\mu'}(U'))$ is a convergence group on $\overline{U'} \cap \Kr$, then
$(\mu,\Kr,\Ur)$ is a convergence $G$-complex too.
\end{lemma}

\begin{proof} Let $g_n \in G$ be a sequence of different elements. We need to show that there are points $a,b \in \Sp^2$ such that (after passing to a subsequence if necessary)  $\mu(g_n) \to a$ uniformly on compact subsets of $\Kr \setminus \{b\}$. 
By the assumption, we know that $\mu'(G)$ is a convergence group on $\Kr'$ so it follows that there are points $a,b \in \Sp^2$ such that (after passing to a subsequence which we also denote by $g_n$) 
$\mu'(g_n) \to a$ uniformly on compact subsets of $\Kr' \setminus \{b\}$. Then the sequence of inverse maps also converges and we have $(\mu'(g_n))^{-1}= \mu'(g^{-1}_n) \to b$ uniformly on compact subsets of $\Kr' \setminus \{a\}$.

Since $\Kr'\subset \Kr$ and $\mu(G)=\mu'(G)$ on $\Kr'$, we conclude that  $\mu(g_n) \to a$, uniformly on compact subsets of  $\Kr' \setminus \{b\}$. Similarly,
$\mu(g^{-1}_n) \to b$, uniformly on compact subsets of $\Kr' \setminus \{a\}$. Hence, for a given compact set $C \subset \Kr \setminus \{b\}$  we know that for every $\epsilon>0$ there exists $n'(\epsilon,C)$ such that 
\begin{equation}\label{eq-rev-0}
d(\mu(g_n)(x),a)<\epsilon, \, \text{for}\,    n>n'(\epsilon,C) \,  \text{and}\,  x \in C \cap \Kr'. 
\end{equation}
Similarly, for a given compact set $D \subset \Kr \setminus \{a\}$  we know that for every $\delta>0$ there exists $n''(\delta,D)$ such that
\begin{equation}\label{eq-rev-0'}
d(\mu(g_n)(x),b)<\delta, \, \text{for}\,    n>n''(\delta,D) \,  \text{and}\,  x \in D \cap \Kr'. 
\end{equation}

It remains to show that there is a function $n(\epsilon,C)$ such that 
\begin{equation}\label{eq-0}
d(\mu(g_n)(x),a)<\epsilon,\,\, \text{for}\,\, n>n(\epsilon,C)\,\, \text{and every}\,\, x \in C.
\end{equation}

\vskip .3cm
For $0<\alpha<1$, let $\overline{\D}^{3}(\alpha)$ denote the closed ball of radius $\alpha$ (with the same center as $\D^3$). Let $C_{\alpha}$ be the compact set defined by
$$
C_{\alpha}\cap \D^{3}=K \cap \overline{\D}^{3}(\alpha),
$$ 
and 
$$
C_{\alpha}\cap \Sp^2=\Sp^{2} \setminus \{ \text{a ball of spherical radius}\, (1-\alpha) \, \text{centered at}\, b \}. 
$$
The collection $C_{\alpha}$ is an exhaustion of $K \setminus \{b\}$ by compact sets. Thus, it is enough to prove (\ref{eq-0}) for each $C_{\alpha}$. From now on we assume that $C=C_{\alpha}$ for some $\alpha \in (0,1)$.
Observe that each such $C$ has the property that if $C \cap (U' \cap \Kr)$ is non-empty for some $U' \in \Ur'$, then either $C \cap \pt{U'}\cap \D^3$ is non-empty or $U'=\D^3$ and thus $\pt{U'}=\Sp^2$. 
(The special choice of compact sets $C$ was made for this reason alone.) 

We similarly define the sets $D_{\alpha}$  that give an exhaustion of $K \setminus \{b\}$ by compact sets). 
\vskip .3cm

Fix $\epsilon>0$.  By the definition of a generalized cell decomposition, there are at most finitely many sets $W_1,..,W_k \in \Ur'$ whose diameter is $\ge \epsilon/2$ (here $k=k(\epsilon)$). 
We have the following claim.

\begin{claim} There exist $M(\epsilon,C), N(\epsilon,C) \in \N$, and sets $V_j \in \Ur'$, $j=1,...,M(\epsilon,C)$, with the following properties.  Let $V \in \Ur'$ and assume that
\begin{itemize}
\item $\overline{V}$ intersects the compact set $C$ (that is $\overline{V} \cap C \ne \emptyset$). 
\item There exists $n>N(\epsilon,C)$ such that $\mu(g_n)(V)=W_i$, for some $i=1,...,k$.
\end{itemize}
Then $V=V_j$ for some  $j$. 
\end{claim}
\begin{proof} Since $C$ is compact and it does not contain $b$, it follows that $C$ does not intersect a sufficiently small open ball around $b$. Let $\delta=\delta(\epsilon,C)$ denote the radius of this ball. 
We define the sets $V_j$ in a similar way we defined $W_i$'s. By the definition of a generalized cell decomposition, there are at most finitely many sets $V_1,..,V_M \in \Ur'$ whose diameter is $\ge \delta/2$ (here $M=M(\epsilon,C)$ depends only on $W_i$'s and $C$). 

Fix a compact set  $D \subset K \setminus \{a\}$  that intersects each $W_i$ (we can choose  $D=D_{\alpha}$ for some $\alpha=\alpha(\epsilon) \in (0,1)$).  Firstly, from (\ref{eq-rev-0'}) it follows that
$$
d(\mu(g^{-1}_n)(x),b)<\delta/2, 
$$
for every $n>n''(\delta/2,D)$ and  $x \in D$ (the constant $n''(\delta/2,D)$ was defined in (\ref{eq-rev-0}) above). The constants $\delta$ and $D$ depend only on $\epsilon$ and $C$ and we set 
$N(\epsilon,C)=n''(\delta/2,D)$.

Since $D$ intersects each $W_i$ it follows that $D$ intersects $\pt{W}_i$, that is there exists $p_i \in \pt{W}_i \cap D$. 
Then, because $\pt{W}_i\subset \Kr'$ it follows $p_i \in \Kr'\cap D$ and thus
the inequality 
\begin{equation}\label{eq-revi-1}
d(\mu(g^{-1}_{n})(p_i),b)<\delta/2, \, \text{holds for}\, n>N(\epsilon,C),
\end{equation}
and some $p_i \in \pt{W}_i$.

Suppose now that that $V \in \Ur'$ and that $\mu(g_n)(V)=W_i$, for some $n>N(\epsilon,C)$. If $V$ is not one of the sets $V_j$ then 
$$
\diam(V)=\diam\big( \mu(g^{-1}_n)(W_i) \big) \le \delta/2.
$$
Together with (\ref{eq-revi-1}), this yields the inequality
$$
d(V,b)=d(\mu(g^{-1}_n)(W_i),b)<\delta,
$$
and therefore $V$ does not intersect $C$. This proves the claim.

\end{proof}

Fix $x \in C \setminus \Kr'$. Then $x \in V$ for some $V \in \Ur'$. We partition the sequence $g_n$ into $(k+1)$ new sequences $g^{i}_n$, $0 \le i \le k$ as follows. By $g^{0}_n$ we denote the subsequence of $g_n$ such that 
\begin{equation}\label{eq-1}
\diam(\mu(g^{0}_n)(\pt{V})) \le \frac{\epsilon}{2}. 
\end{equation}
By $g^{i}_n$ we denote the subsequence  of $g_n$ such that $g^{i}_n(V)=W_i$. At least one the sequences $g^{i}_m$ is infinite but some may be finite (or empty).

\begin{remark} Clearly, the element $g^{i}_{n}$ is not necessarily defined for every $n \in \N$, so we consider the sequence $g^{i}_{n}$  as a sequence indexed by the corresponding subset of $\N$ which may be finite.
\end{remark}

The following is an important observation and we formulate it as a claim.

\begin{claim} If $V$ is not one of the sets $V_j$ from the previous claim then for each  $1\le i \le k$ the corresponding sequence $g^{i}_n$ is vacuous for $n >N(\epsilon,C)$.
\end{claim}
\begin{proof}  If $V$ is not one of the sets $V_j$ and $n >N(\epsilon,C)$ then $\mu(g_n)(V)$ is disjoint from any $W_i$. This proves the claim. 
\end{proof} 
\vskip .3cm
Recall that we fix  $x \in C \setminus \Kr'$, and that $V \in \Ur'$ is such that $x \in V$.  From (\ref{eq-rev-0}) it follows that
$$
d(\mu(g_n)(x),a)<\epsilon/2, 
$$
for every $n>n'(\epsilon/2,C)$. Since $C$ intersects $V$ there exists $q \in \pt{V} \cap C$. Then, because $\pt{V} \subset \Kr'$ it follows $q \in \Kr'\cap C$ and thus
the inequality 
\begin{equation}\label{eq-revi-1'}
d(\mu(g_{n})(q),a)<\epsilon/2, \, \text{holds for}\, n>n'(\epsilon/2,C),
\end{equation}
and some $q \in \pt{V}$.

On the other hand, by definition we have
$$
\diam\big( \mu(g^{0}_n)(V) \big) \le \epsilon/2.
$$
Together with (\ref{eq-revi-1'}), this yields the inequality
$$
d(\mu(g^{0}_n(V),a)<\epsilon,
$$
and therefore we have 
\begin{equation}\label{eq-revi-2}
d(\mu(g^{0}_n)(x),a)<\epsilon, \,\, \text{for every}\,\,  n>n'(\epsilon/2,C).
\end{equation}

\vskip .3cm
Fix $1 \le i \le k$. We will show that there exists $n_{i}(\epsilon,C,V)$ such that 
\begin{equation}\label{eq-rev-1}
d(\mu(g^{i}_n)(x),a)<\epsilon, \, \text{for}\, n>n_{i}(\epsilon,C,V).
\end{equation}
\vskip .3cm
\begin{remark} Note that if $V$ is not one of the sets $V_j$ from the first claim above then $g^{i}_n$ does not exist for $n>N(\epsilon,C)$. Thus for each such $V$ we have that (\ref{eq-rev-1}) holds for  $n_{i}(\epsilon,C,V)=N(\epsilon,C)$.
\end{remark}
\vskip .3cm
Set
$$
h^{i}_{n}=g^{i}_n \circ (g^{i}_{n_{i}})^{-1},
$$
where $n_{i}$ is the smallest number for which the element $g^{i}_{n}$ is defined. Then $\mu(h^{i}_{n})(W_i)=W_i$ and so $h^{i}_{n} \in \stab_{\mu'}(W_i)$. 
Recall that $g_n \to a$ on compact sets in $\Kr' \setminus \{b\}$. In terms of the sequence $h^{i}_n$ this means that
\begin{equation}\label{eq-revis}
h^{i}_n(z) \to a \, \,\text{uniformly on compact sets in} \,  \overline{W}_{i} \cap (\Kr' \setminus \{b_i\}), 
\end{equation}
where $b_i=\mu(g^{i}_{n_{i}})(b)$.

Let  $x_i=\mu(g^{i}_{n_{i}})(x)$. The claim (\ref{eq-rev-1}) can be rewritten as the following: There exists $n_{i}(\epsilon,C,V)$ such that
\begin{equation}\label{eq-rev-2}
d(\mu(h^{i}_n)(x_i),a)<\epsilon, \,\, \text{for}\, n>n_{i}(\epsilon,C,V).
\end{equation}

Suppose that (\ref{eq-rev-2}) does not hold for any $n_i(\epsilon,C,V)$. Then there exist an infinite subsequence $h^{i}_{m}$ of $h^{i}_{n}$, and a sequence of points $z_m \in \mu(g^{i}_{n_{i}})(C) \cap \overline{W}_i$,  such that 
\begin{equation}\label{eq-rev-3}
d(\mu(h^{i}_m)(z_m),a) \ge \epsilon.
\end{equation}

By the assumption of the lemma, the group $\mu(\stab_{\mu'}(W_i))$ is a convergence group on $\overline{W_i} \cap \Kr$. Thus,  after passing to a subsequence of $h^{i}_{m}$ if necessary, we find that 
there exist  points  $a', b' \in \overline{W_i} \cap \Kr$, such that the sequence $h^{i}_{m}$ converges to $a'$ uniformly on compact subsets of $\overline{W_i} \cap (\Kr \setminus \{b'\})$. 

On the other hand, from (\ref{eq-revis}) we get that $h^{i}_n(z) \to a$ uniformly on compact sets in $\pt{W}_{i} \setminus \{b_i\}$ (relying on the inclusion  $\pt{W}_{i} \subset \overline{W}_{i} \cap \Kr'$). Since  $\pt{W}_{i} \setminus \{b_i\}$ is non-empty (the boundary of a set from $\Ur'$ contains more than one point), it follows that $a=a'$.  Therefore $h^{i}_{m}$ converges to $a$ uniformly on compact subsets of  $\overline{W_i} \cap (\Kr \setminus \{b_i\})$, and in particular, the  sequence 
$h^{i}_{m}$ converges to $a$ uniformly on $\mu(g^{i}_{n_{i}})(C) \cap \overline{W}_i$. But this contradicts (\ref{eq-rev-3}) and we have shown that (\ref{eq-rev-2}) holds for some $n_i(\epsilon,C,V)$.

Above we proved that for $n>n_{i}(\epsilon,C,V)$ and $x \in C \cap \overline{V}$ we have $d(\mu(g^{i}_n)(x),a)<\epsilon$. 
Let 
$$
n^{j}(\epsilon)=\max\{n'(\epsilon/2,C), N(\epsilon,C), n_1(\epsilon,C,V_j)\},
$$
and $n(\epsilon)=\max \{n^{1}(\epsilon),...,n^{M(\epsilon,C)}(\epsilon))\}$. Then (\ref{eq-0}) holds for this choice of $n(\epsilon)$ and we are finished.

\end{proof}

\subsection{Building new $G$-complexes out of old}  Let  $(\mu_1,\Kr_1,\Ur_1,\phi^{\Ur_{1}}_{1})$ and $(\mu_2,\Kr_2,\Ur_2,\phi^{\Ur_{2}}_{2})$ denote two marked $G$-complexes. We assume that 
\begin{equation}\label{eq-ass}
\mu_1(g)(x)=\mu_2(g)(x), \,\, \text{for every} \,\, x \in \Sp^2, \,\, \text{and}\,\, g \in G.
\end{equation}
We define a new marked $G$-complex $(\mu,\Kr,\Ur,\phi^{\Ur})$ out of these two $G$-complexes as follows. 

Let
$$
\Ur=\{ U: \, U=\phi^{V}_{1}(W), \, \text{where}\, V \in \Ur_1\, \text{and}\, W \in \Ur_2\},
$$
and define the corresponding markings $\phi^U:(\overline{\D}^3,\Sp^2) \to (\overline{U}, \pt{U})$, $U \in \Ur$, by $\phi^U=\phi^{V}_{1} \circ \phi^{W}_{2}$ (observe that each 
$U \in \Ur$ is uniquely written as $U=\phi^{V}_{1}(W)$). 

The closed set $\Kr$ is defined as the complement of the union of sets from $\Ur$. Another way to describe $\Kr$ is 
$$
\Kr=\Kr_1 \cup \left( \bigcup_{V \in \Ur_1}   \phi^{V}_{1}(\Kr_2) \right).
$$
The pair $(\Kr,\Ur)$ is a generalized cell decomposition of $\D^3$

Next, we define the $G$-action $\mu:G \to \Homeo(\Kr)$  as follows.  By definition, $\Kr_1 \subset \Kr$. We set $\mu(g)=\mu_1(g)$ on $\Kr_1$. It remains to define $\mu(g)$ on $\Kr \cap V$ for each $V \in \Ur_1$. 
We let 
\begin{equation}\label{eq-rule}
\mu(g)=   \phi^{\mu_{1}(g)(V)}_{1}   \circ \mu_2(g)  \circ \big( \phi^{V}_{1} \big)^{-1}, \,\,\text{on}\,\, \Kr \cap \overline{V}.
\end{equation}
\vskip .3cm
We need to check that $\mu(g)$ is a homeomorphisms on $\Kr$. By definition, the restrictions of $\mu(g)$ on $\Kr_1$ and  $\Kr \cap V$ respectively are homeomorphisms. But we gave two definitions of $\mu(g)$ on $\pt{V}$. One one hand, we have 
$\mu(g)=\mu_1(g)$ since $\pt{V}\subset \Kr_1$, and on the other hand $\mu(g)$ is defined on $\pt{V}$ by (\ref{eq-rule}). We need to verify the equality
$$
\mu_1(g)= \phi^{\mu_{1}(g)(V)}_{1}   \circ \mu_2(g)  \circ \big( \phi^{V}_{1} \big)^{-1}, \,\,\text{on}\,\, \pt{V}.
$$
By  the assumption (\ref{eq-ass}) we have $\mu_2(g)=\mu_1(g)$ on $\Sp^2$. Replacing this in the previous equality yields  
$$
\mu_1(g)= \phi^{\mu_{1}(g)(V)}_{1}   \circ \mu_1(g)  \circ \big( \phi^{V}_{1} \big)^{-1}, \,\,\text{on}\,\, \pt{V}.
$$
This holds  by the property $(1)$ from the definition of a $G$-complex.

\vskip .3cm
\begin{remark} This is one of the main points in our construction of the $G$-complex $(\mu,\Kr,\Ur)$.  This explains  why it is natural to assume (\ref{eq-ass}). 
Moreover,  this is why we require in the definition of a $G$-complex that maps  $\phi^{U}:\Sp^2 \to \pt{U}$, $U \in \Ur$,  conjugate the action of 
$\mu(g)$ from $\Sp^2$ to its action between the boundaries of sets from $\Ur$. This is a strong requirement and it means that in general it will be difficult to extend a convergence action on $\Sp^2$ to a non-trivial convergence $G$-complex.
\end{remark}
\vskip .3cm

We show that $\mu:G \to \Homeo(\Kr)$ is a homomorphism. For $f,g \in G$, one verifies (recall the convention $\mu(fg)=\mu(g)\circ \mu(f)$):
\begin{align*}
\mu(fg) &= \phi^{\mu_{1}(fg)(V)}_{1}   \circ \mu_2(fg)  \circ \big( \phi^{V}_{1} \big)^{-1} \\
&=\phi^{\mu_{1}(g)(\mu_{1}(f)(V))}_{1}   \circ \mu_2(g) \circ  \mu_2(f)  \circ \big( \phi^{V}_{1} \big)^{-1} \\
&=\left( \phi^{\mu_{1}(g)(\mu_{1}(f)(V))}_{1}   \circ \mu_2(g)\circ  \big( \phi^{\mu_{1}(f)(V)}_{1} \big)^{-1} \right) \circ \left(\phi^{\mu_{1}(f)(V)}_{1} \circ  \mu_2(f)  \circ \big( \phi^{V}_{1} \big)^{-1} \right) \\
&= \mu(g)\circ \mu(f),
\end{align*}
thus $\mu$ is a $G$-action and therefore the marked $G$-complex $(\mu,\Kr,\Ur,\phi^{\Ur})$ is well defined. Clearly the $G$-complex $(\mu,\Kr,\Ur)$ is a refinement of $(\mu_1,\Kr_1,\Ur_1)$.

We say that $(\mu,\Kr,\Ur)$ is the refinement of  $(\mu_1,\Kr_1,\Ur_1)$ induced by $(\mu_2,\Kr_2,\Ur_2)$. The next proposition gives a sufficient condition for  the new $G$-complex $(\mu,\Kr,\Ur)$ to be  a convergence $G$-complex.

\begin{proposition}\label{prop-triv-0} Let  $(\mu_i,\Kr_i,\Ur_i)$, $i=1,2$, denote two  $G$-complexes and let $(\mu,\Kr,\Ur)$ be the refinement of  $(\mu_1,\Kr_1,\Ur_1)$ induced by $(\mu_2,\Kr_2,\Ur_2)$. Suppose that  $(\mu_1,\Kr_1,\Ur_1)$ is a convergence $G$-complex and that for every $U_1 \in \Ur_1$ the group  $\mu_2(\stab_{\mu_{1}}(U_1))$ is a convergence group on $\Kr_2$ (the latter condition is satisfied if  $(\mu_2,\Kr_2,\Ur_2)$ is a convergence $G$-complex). Then $(\mu,\Kr,\Ur)$ is a convergence $G$-complex.
\end{proposition}

\begin{proof} Fix $U_1 \in \Ur_1$. Then $\mu(\stab_{\mu_{1}}(U_1))$ acts on $\overline{U}_1 \cap \Kr$ as a convergence group because it is conjugate (by some marking  $\phi^{U_{1}}_{1}:\Sp^2 \to \pt{U}_1$)
to the action of the group $\mu_2(\stab_{\mu_{1}}(U_1))$ on $\Kr_2$,  and the latter is a convergence group by the assumption. Combining this with the assumption that  $(\mu_1,\Kr_1,\Ur_1)$ is a convergence $G$-complex and  Lemma \ref{lemma-glava} we conclude that $(\mu,\Kr,\Ur)$ is a convergence $G$-complex.
\end{proof}

The following proposition is elementary and its proof is left to the reader.

\begin{proposition}\label{prop-triv} Let $(\mu,\Kr,\Ur)$ be the refinement of  $(\mu_1,\Kr_1,\Ur_1)$ induced by $(\mu_2,\Kr_2,\Ur_2)$. Suppose that  $(\mu_1,\Kr_1,\Ur_1)$ is a free $G$-complex and that for every $U_1 \in \Ur_1$, the non-trivial elements in the group  $\mu_2(\stab_{\mu_{1}}(U_1))$ have no fixed points in  $\D^3 \cap \Kr_2$ (the latter condition is satisfied if $(\mu_2,\Kr_2,\Ur_2)$ is a free $G$-complex). Then $(\mu,\Kr,\Ur)$ a free $G$-complex.
\end{proposition}

One may think of the $G$-complex $(\mu,\Kr,\Ur)$ as the intersection between  $(\mu_1,\Kr_1,\Ur_1)$ and $(\mu_2,\Kr_2,\Ur_2)$. This is illustrated by the following proposition which is the main motivation behind defining the new $G$-complex $(\mu,\Kr,\Ur)$
from the old ones $(\mu_1,\Kr_1,\Ur_1)$ and $(\mu_2,\Kr_2,\Ur_2)$.

\begin{proposition}\label{prop-inter} For every $U \in \Ur$ there are $U_i \in \Ur_i$ such that
\begin{equation}\label{eq-2}
\stab_{\mu}(U)=\stab_{\mu_{1}}(U_1) \cap \stab_{\mu_{2}}(U_2) 
\end{equation}

\end{proposition}

\begin{proof} Let $U_i \in \Ur_i$ be such that
$$
U=\phi^{U_{1}}_{1}(U_2).
$$
Then (\ref{eq-2}) holds by definition.
\end{proof}

We end this subsection by showing that if we have a $G$-complex $(\mu,\Kr,\Ur)$ such that $\stab_{\mu}(U)$ is the trivial group for every $U \in \Ur$, then we can extend the action  $\mu:G \to \Homeo \Kr$ to a convergence action of $G$ on $\overline{\D}^3$.

\begin{proposition}\label{prop-vazno} Let $(\mu,\Kr,\Ur)$ be a free convergence $G$-complex such that $\stab_{\mu}(U)$ is the trivial group for every $U \in \Ur$. Let  $\Kr_0=\overline{\D}^3$ and let $\Ur_0$ denote an empty collection. 
Then  there exists a free and convergence $G$-complex $(\mu_0,\Kr_0,\Ur_0)$ that is a refinement of the $G$-complex  $(\mu,\Kr,\Ur)$. (In other words, under these assumptions the action of the convergence group $\mu(G)$ acting on $\Sp^2$ can be extended to a free convergence action on $\D^3$.)
\end{proposition}

\begin{proof} We construct the $G$-complex $(\mu_0,\Kr_0,\Ur_0)$ as the refinement of $(\mu,\Kr,\Ur)$ by the $G$-complex $(\mu_{\text{rad}},\Kr_0,\Ur_0)$ which we define as follows. 

Let $\mu_{\text{rad}}(g):\overline{\D}^{3} \to \overline{\D}^3$ be the radial extension of the homeomorphism $\mu(g):\Sp^2 \to \Sp^2$. The radial extension defines a monomorphism from \newline
$\Homeo(\Sp^2) \to \Homeo(\overline{\D}^3)$ and thus $(\mu_{\text{rad} },\Kr_0,\Ur_0)$ is a $G$-complex (note that this $G$-complex is neither free nor convergence). Since $\stab_{\mu}(U)$ is trivial for every $U \in \Ur$ by Proposition \ref{prop-inter}, 
it follows from Proposition \ref{prop-triv} the $G$-complex $(\mu_0,\Kr_0,\Ur_0)$ is a free $G$-complex  and by Proposition \ref{prop-triv-0}  that it is a convergence $G$-complex.

\end{proof}

\subsection{$G$-complex arising from malnormal surface subgroups of $G$} From now on, we assume that $G$ is a torsion-free hyperbolic group whose boundary $\pt{G}$ is homeomorphic to $\Sp^2$, and such that elements of $G$ act as orientation preserving homeomorphisms on  $\pt{G}\approx \Sp^2$. We fix a homeomorphism $\Psi_0:\pt{G} \to \Sp^2$ and thus obtain a convergence $G$-action $\mu_0:G \to \Homeo(\Sp^2)$ by $\mu_0(g)=\Psi_0 \circ g \circ \Psi^{-1}_{0}$.

\begin{remark} Note that  $(\mu_0,\Sp^2,\D^3)$ is a convergence $G$-complex (the corresponding collection $\Ur$ contains only one element $\D^3$).
\end{remark}

In what follows we assume that $H_0<G$ is a quasi-convex,  malnormal surface subgroup. By $\mathcal{H}$ we denote the collection of $G$-conjugates of $H_0$.  
For each $H \in \mathcal{H}$, by $\gamma_H \subset \Sp^2$ we denote the limit set of $H$ (the set $\gamma_H$ is a Jordan curve because $H$ is a quasi-convex subgroup). 
The  curves $\gamma_H$ are disjoint for different $H$ because $H_0<G$ is malnormal.

The surface group $\mu_0(H)$ is acting on the Jordan curve $\gamma_H$ as a convergence group. It follows from Theorem \ref{thm-GCJ} of Gabai and Casson-Jungreis that the action of $\mu_0(H)$ on $\gamma_H$ 
is conjugated to a Fuchsian surface group. This implies that every element of  $\mu_0(H)$ preserves a chosen orientation on $\gamma_H$.  
Thus we can define the  orientation on each $\gamma_H$ such that $\mu_0(g)(\gamma_H)=\gamma_{g^{-1}Hg}$  as oriented curves for every $g \in G$.

The set $\Sp^2 \setminus \gamma_H$   has two components (both Jordan domains). The one to the right of $\gamma_H$ we denote  by $D_H$ and the one to the left by $L_H$ (since $\mu_0(g)$, $g \in H$, preserves an orientation on $\gamma_H$ it follows that $\mu_0(g)(D_H)=D_H$ and similarly for $L_H$).   The following proposition will be used below in the proof of Lemma \ref{lemma-imp}.

\begin{proposition}\label{prop-nada} There exists a collection of homeomorphisms $\xi_H:\overline{D}_H \to \overline{L}_H$, $H \in \mathcal{H}$, with the following properties:
\begin{enumerate}
\item $\xi_H=\id$ on $\pt{D}_H=\pt{L}_H=\gamma_H$.
\item The equality
\begin{equation}\label{eq-0-l}
\big( \xi_{g^{-1}Hg} \circ \mu_0(g) \big)(x)=\big( \mu_0(g) \circ \xi_H \big)(x), 
\end{equation}
holds  for each $g \in G$,  $x \in D_H$, and $H \in \mathcal{H}$.
\end{enumerate}

\end{proposition}

\begin{proof} The surface group $\mu_0(H_0)$ acts on both $D_{H_{0}}$ and $L_{H_{0}}$ as a convergence group. It is well known (see \cite{tukia}, \cite{zei}) that a convergence group acting on a Jordan domain is topologically conjugated to a Fuchsian group acting on the unit disc $\D^2$. That is, there are homeomorphisms $\eta_D: \overline{D_{H_{0}}} \to \overline{\D}^2$ and $\eta_L: \overline{L_{H_{0}}} \to \overline{\D}^2$ such that $\eta_D\mu(H_0)\eta^{-1}_D$ and $\eta_L\mu(H_0)\eta^{-1}_L$ 
are Fuchsian groups. 

The circle homeomorphism $\eta_L \circ \eta^{-1}_D:\Sp^1 \to \Sp^1$ conjugates the group $\eta_D\mu(H_0)\eta^{-1}_D$ to $\eta_L\mu(H_0)\eta^{-1}_L$. Let $\eta:\overline{\D}^2 \to \overline{\D}^2$ be an equivariant extension of  
$\eta_L \circ \eta^{-1}_D$. 
Set 
$$
\xi_{H_{0}}= \eta^{-1}_L \circ \eta  \circ \eta_D.
$$
Then $\xi_{H_{0}}:\overline{D}_{H_{0}} \to \overline{L}_{H_{0}}$, is such that $\xi_{H_{0}}=\id$ on $\gamma_{H_{0}}$ and 
\begin{equation}\label{eq-1-l}
\xi_{H_{0}}^{-1}\circ \mu_0(h) \circ \xi_{H_{0}}=\mu_0(h),\,\, \text{for}\,\, h \in H_0.
\end{equation}
\vskip .3cm
For $g \in G$ and $H=g^{-1}H_0 g$,  we let 
\begin{equation}\label{eq-rev-a}
\xi_{H}=\mu_0(g) \circ \xi_{H_{0}}  \circ \mu^{-1}_{0}(g), \,\, \text{on}\,\, D_H.
\end{equation}
We need to check that this definition is independent of the choice of $g$ for which  $H=g^{-1}H_0 g$.
Let $f \in G$ such that $H=f^{-1}H_0 f$. We need to check the equality 
$$
\mu_0(f) \circ \xi_{H_{0}}  \circ \mu^{-1}_{0}(f)=\mu_0(g) \circ \xi_{H_{0}}  \circ \mu^{-1}_{0}(g).
$$
This equality can be rewritten as 
$$
\mu^{-1}_{0}(g) \circ \mu_0(f) \circ \xi_{H_{0}}  \circ \mu^{-1}_{0}(f)\circ \mu_0(g)=\xi_{H_{0}},
$$
and it becomes 
$$
\mu_0(fg^{-1}) \circ \xi_{H_{0}}  \circ \mu^{-1}_{0}(fg^{-1})=\xi_{H_{0}}.
$$
Let $h=fg^{-1}$. Then $hH_0 h^{-1}=H_0$ and it follows that $h \in H_0$ since $H_0$ is malnormal. Then the last equality is equivalent to (\ref{eq-1-l}). 

The equality (\ref{eq-0-l}) is straightforward to verify  and we leave  this to the reader.

\end{proof}

To each quasi-convex, malnormal surface subgroup $H_0<G$ we  associate a generalized cell decomposition as follows.

\begin{proposition}\label{prop-framing} Let $H_0<G$ denote a quasi-convex, malnormal surface subgroup and let $\mathcal{H}$ denote the collection of all $G$-conjugates of $H_0$.
There exists a collection of embedded discs $B_H \subset \D^3$, $H \in \mathcal{H}$,  with the following properties:
\begin{enumerate}
\item The boundary $\pt{B}_H$ equals $\gamma_H$.
\item The closed  discs $\overline{B}_H \subset \overline{\D}^3$ are mutually disjoint for different $H \in \mathcal{H}$.
\item If $\diam(\gamma_H) \to 0$ (along some sequence $H \in \mathcal{H}$) then $\diam(B_H) \to 0$ also (here $\diam$ refers to the standard metric on $\R^3$).
\item Let 
$$
\Kr_{\mathcal{H}}=\Kr=\Sp^2 \cup \big(\bigcup_{H \in \mathcal{H}} B_H  \big),
$$
and $\Ur_{\mathcal{H}}=\Ur$ the collection of connected components of the complement $\overline{\D}^3 \setminus \Kr$. Then $\Kr$ is a closed set and each  pair $(\overline{U},\pt{U})$, $U \in \Ur$,  is homeomorphic to $(\overline{\D}^3,\Sp^2)$.
\item Let $U \in \Ur$ and $a,b \in \pt{U}\cap \Sp^2$. Then no Jordan curve $\gamma_H$, $H \in \mathcal{H}$, separates the points $a$ and $b$.
\end{enumerate}
\end{proposition}
\begin{proof}  We first prove that there exists a homeomorphism $\Psi_0:\pt{G} \to \Sp^2$  such that $\gamma_H$ are (geometrically) round circles in $\Sp^2$. The following theorem (which we prove in the Appendix) is a 
slight extension of the classical result of Whyburn that states that every two Sierpinski curves are homeomorphic.

\begin{theorem}\label{thm-whyburn}  Let $\gamma_k$, $k \in \N$ denote a null sequence of Jordan curves in $\Sp^2$. Then there exists a homeomorphism $F:\Sp^2 \to \Sp^2$ such that  each  curve $F(\gamma_k)$ is a round  circle.
\end{theorem}

Consider the sequence of curves $\gamma_H$, $H \in \mathcal{H}$. Then for each $\epsilon>0$ there are only finitely many curves $\gamma_H$ whose diameter is $>\epsilon$ (see Lemma \ref{lemma-GMRS}). Thus the sequence of curves $\gamma_H$, $H \in \mathcal{H}$, is a null sequence (note that the union of curves $\gamma_H$ is dense in $\pt{G}$).  We apply Theorem \ref{thm-whyburn} and conclude that there exits a homeomorphism $\Psi_0:\pt{G} \to \Sp^2$  such that $\gamma_H$ are (geometrically) round circles in $\Sp^2$. From now on we assume that $\gamma_H$ are round circles.

Let $B_H \subset \overline{\D}^3$ denote the hyperbolic disc (with respect to the hyperbolic metric on $\D^3$) that bounds $\gamma_H$. 
The discs $B_H$ are disjoint for different $H \in \mathcal{H}$ and $\diam(B_H)=\diam(\gamma_H)$ (here $\diam$ refers to the metric inherited from $\R^3$). 

Each connected component $U \in \Ur$ is a convex (with respect to the hyperbolic metric)  subset of $\D^3$ and it follows that $(\overline{U},\pt{U})$ is homeomorphic to $(\overline{\D}^3,\Sp^2)$.  

Let $a,b \in \pt{U}\cap \Sp^2$. Since $U$ is convex (in the hyperbolic metric), the hyperbolic geodesic $\beta$ between $a$ and $b$ is contained in $U$. Suppose that $a$ and $b$ are separated by $\gamma_H$ for some $H \in \Ur_H$. Then $\beta$ intersects the disc $B_{H}$. We first show that $B_H$ is then necessarily contained in $U$. Let $p=\beta \cap B_H$ and denote by $s \ge 0$ the supremum of the set of all non-negative numbers $r$ such that the disc of radius $r$ in (the hyperbolic plane) $B_H$ centered at $p$ is contained in $U$. Since $U$ is open and $p\in U \cap B_H$ we have $s>0$. Suppose that $s<+\infty$. Then the disc of radius $s$ around $p$ contains a point from $\pt{U}$. Since this point does not belong to $\Sp^2$, it has to belong to some disc $B_{H'}$ that is in the boundary of $U$. Thus $B_H \cap B_{H'}\ne \emptyset$ and this implies $B_H=B_{H'}$. But this is a contradiction since $p \in U$ and therefore $p$ does not belong to $B_{H'}\subset \pt{U}$.

Since $B_H \subset U$ it follows that  $B_{H}$ separates $U$ into two connected components and this contradicts the assumption that $U$ is connected. Thus $a$ and $b$ are not separated by any $\gamma_H$.

\end{proof}
It follows that $(\Kr_{\mathcal{H}},\Ur_{\mathcal{H}})$ is a generalized cell decomposition. We now promote it to a free convergence $G$-complex. The following is the main result of this subsection.

\begin{lemma}\label{lemma-imp} There exists a free, convergence $G$-complex $(\mu_{\mathcal{H}},\Kr_{\mathcal{H}},\Ur_{\mathcal{H}})$ with the following property: Let $U \in \Ur$ and (as before) let $\stab_{\mu_{\mathcal{H}}}(U)<G$ denote the stabilizer of $U$ under the action $\mu_{\mathcal{H}}$. Suppose $f \in \stab_{\mu_{\mathcal{H}}}(U)$, $f \ne \id$, and let $f^{+}$ and $f^{-}$ denote the fixed points of $\mu_{\mathcal{H}}(f)f$. 
Then no Jordan curve $\gamma_H$, $H \in \mathcal{H}$,  separates the points $f^{+}$ and $f^{-}$.
\end{lemma}

\begin{proof} We define the $G$-action $\mu_{\mathcal{H}}=\mu:G \to \Homeo(\Kr)$ as follows (where $\Kr=\Kr_{\mathcal{H}}$ was defined in the previous proposition). We let $\mu=\mu_0$ on $\Sp^2$. 

For each $H \in \mathcal{H}$ we fix a homeomorphism $\psi_H: \overline{L}_H \to \overline{B}_H$, such that $\psi_H=\id$ on $\gamma_H$. We let
\begin{equation}\label{eq-def}
\mu(g)(x)=\big( \psi_{g^{-1}Hg} \circ \mu_0(g) \circ \psi^{-1}_H \big)(x), \,\, \text{for}\,\, x \in \overline{B}_H.
\end{equation}

It is straightforward to verify that $\mu(g) \in \Homeo(\Kr)$ and that $\mu:G \to \Homeo(\Kr)$ is a $G$-action.  
To show that $(\mu,\Kr,\Ur)$ is a $G$-complex it remains to construct homeomorphisms $\phi^U:(\overline{\D}^3,\Sp^2) \to (\overline{U},\pt{U})$, for $U \in \Ur$, such that $\phi^U$ conjugates the action of $\mu(G)$ 
(meaning that the condition $(1)$ from Definition \ref{def-rev} holds).

Fix $U \in \Ur$. We let $\phi^U=\id$ on $\pt{U} \cap \Sp^2$.  Observe that the set $\pt{U} \setminus \Sp^2$ is a disjoint union of discs $B_H$, where $H$ belongs to a sub-collection $\mathcal{H}_U \subset \mathcal{H}$. 
Moreover, every connected component of the set $\Sp^2 \setminus \pt{U}$ is either $D_H$ or $L_H$ for some $H \in \mathcal{H}_U$.  Thus, we can write the set 
$\Sp^2 \setminus \pt{U}$ as a disjoint union of the corresponding sets $D_H$ or $L_H$.

If $x \in L_H$, for some $H \in \mathcal{H}_U$, we let $\phi^U=\psi_H$. If  $x \in D_H$, for some $H \in \mathcal{H}_U$, we let $\phi^U=\psi_H \circ \xi_H$ (see the definition $\xi_H$ from Proposition \ref{prop-nada}).
Since  $\xi_H$ is the identity on $\gamma_H$ it follows that $\phi^U$ is the identity on $\pt{U}\cap \Sp^2$. 

Next, we verify that $\phi^U$ satisfies the condition $(1)$ from Definition \ref{def-rev}, that is we check the equality 
\begin{equation}\label{eq-rev-b}
\mu_0(g)\circ \phi^U=\phi^{\mu_{0}(g)(U)} \circ \mu_0(g),\, \text{on}\,\, \Sp^2, \,\, \text{for every}\,\,  g \in G.
\end{equation}
We first consider (\ref{eq-rev-b}) for the restriction of $\phi^{U}$ to some $L_H$. In this case $\phi^U=\psi_H$ and (\ref{eq-rev-b})  follows from (\ref{eq-def}). Consider the restriction of $\phi^{U}$ to some $D_H$. Then (\ref{eq-rev-b}) follows from 
(\ref{eq-def})  and from (\ref{eq-0-l}) in  Proposition \ref{prop-nada}. Hence,  we have proved that $\phi^U$ conjugates the action of $\mu(G)$ from $\Sp^2$ to $\pt{U}$.

So far we have defined the map $\phi^U$ to be a a homeomorphism from $\Sp^2$ to $\pt{U}$. We define $\phi^U:\overline{\D}^3 \to \overline{U}$ to be any homeomorphism that that agrees with the map  $\phi^U$ on $\Sp^2$.

We check that $\mu$ is a free $G$-action as follows. Suppose $\mu(g)(x)=x$. Then $x$ belongs to some $B_H$ and $g \in \stab_G(H)$. Since $H$ is a malnormal (and thus maximal) subgroup it follows $g \in H$. Thus, $\mu_0(g)$ fixes the point $\phi_H(x) \in L_H$. We know that $\mu_0(g)$ has two fixed points on $\gamma_H$ and it follows that $\mu_0(g)$ has at least 3 fixed points which implies that $g=\id$ (recall that we assume that $G$ is torsion free and $\mu_0(g)$ are orientation preserving). 

We already know that $\mu(G)$ is a convergence group on $\Sp^2$. It then follows from the property (3) in Proposition \ref{prop-framing} above that $\mu$ is a convergence $G$-complex. Also, let $f \in \stab_{\mu_{\mathcal{H}}}(U)<G$, $f \ne \id$. It follows from  (5) in Proposition \ref{prop-framing} that no Jordan curve $\gamma_H$, $H \in \mathcal{H}$, separates the points $f^{+}$ and $f^{-}$.

\end{proof}

\subsection{The proof of Theorem \ref{thm-1-new}} Let $G$ be a torsion-free hyperbolic group with $\pt{G}\approx \Sp^2$, which acts by orientation preserving homeomorphisms on $\pt{G}$. Suppose that $H_i<G$, $1\le i\le k$, are quasi-convex, malnormal surface subgroups, such that every two points of $\pt{G}$ are separated by the limit set of some $G$-conjugate of some $H_i$. Denote by $\mathcal{H}_{i}$ the collection of $G$-conjugates of $H_i$. Let $(\mu_{\mathcal{H}_{i}},\Kr_{\mathcal{H}_{i}},\Ur_{\mathcal{H}_{i}})$ be the free, convergence $G$-complex from Lemma \ref{lemma-imp}. 

Let $(\mu_{j+1},\Kr_{j+1},\Ur_{j+1})$ be the $G$-complex that is defined inductively as the refinement of the $G$-complex $(\mu_j,\Kr_j,\Ur_j)$ by the $G$-complex  \newline
$(\mu_{\mathcal{H}_{j+1}},\Kr_{\mathcal{H}_{j+1}},\Ur_{\mathcal{H}_{j+1}})$, and $(\mu_1,\Kr_1,\Ur_1)=(\mu_{\mathcal{H}_{1}},\Kr_{\mathcal{H}_{1}},\Ur_{\mathcal{H}_{1}})$. 
Set  $(\mu,\Kr,\Ur)=(\mu_k,\Kr_k,\Ur_k)$. 

Since each  $(\mu_j,\Kr_j,\Ur_j)$ is a free, convergence $G$-complex, it follows from Proposition \ref{prop-triv-0} and Proposition \ref{prop-triv} that  $(\mu,\Kr,\Ur)$ is a free, convergence $G$-complex. Let $U \in \Ur$. Then it follows from  Proposition \ref{prop-inter} that there are sets $U_i \in \Ur_{\mathcal{H}_{i}}$ such that 
$$
\stab_{\mu}(U)=\bigcap_{i=1}^{i=k} \stab_{\mu_{\mathcal{H}_{i}}}(U_i).
$$

Suppose that $g \in \stab_{\mu}(U)$. Then $g \in \stab_{\mu_{\mathcal{H}_{i}}}(U_i)$, for each $i$. If $g \ne \id$,  we see from Lemma \ref{lemma-imp} that the fixed points $g^{+}$ and $g^{-}$ are not separated by 
any $G$-conjugate of $H_i$, $1\le i \le k$. But this contradicts the assumption that every two points in $\pt{G}$ are separated by such a conjugate. Thus each $\stab_{\mu}(U)$ is trivial.

Let $\Kr_0=\D^3$ and $\Ur_0$ an empty collection. From Proposition \ref{prop-vazno} we obtain a free and convergence $G$-complex $(\mu_0,\Kr_0,\Ur_0)$ that is a refinement of the $G$-complex  $(\mu,\Kr,\Ur)$. 
Then, from Proposition \ref{prop-lako} we conclude that $G$ is isomorphic to the fundamental group of a 3-manifold, and we are finished.

\section{Appendix} An $S$-curve (or a Sierpinski curve)  is a subset of the 2-sphere $\Sp^2$ that remains after removing from $\Sp^2$ the interiors of a null sequence of mutually disjoint closed Jordan  disks whose
union is dense in $\Sp^2$. The well known theorem of Whyburn (Theorem 3 in \cite{whyburn}) states that every two $S$-curves are homeomorphic (Cannon \cite{cannon} has extended this result to $(n-1)$-dimensional $S$-curves for $n \ne 4$).

In fact, Whyburn  proves the following (see Theorem 3 and its proof in \cite{whyburn} or see \cite{cannon}). 
We say that a Jordan curve $\gamma$ is an outer boundary curve of an $S$-curve  if $\gamma$ bounds one of the Jordan discs in its complement (if we talk about $\gamma$ as an oriented curve we mean the orientation that $\gamma$ inherits as the boundary of the corresponding Jordan disc).

\begin{theorem}\label{thm-whyburn-1} Let $S_1$ and $S_2$ denote two $S$-curves  and let $\gamma_1$ and $\gamma_2$ denote outer boundary curves of $S_1$ and $S_2$ respectively. 
Let $h_0:\gamma_1 \to \gamma_2$ denote an orientation preserving homeomorphism. Then there exists a homeomorphism  $h:S_1 \to S_2$ that extends $h_0$.
\end{theorem}

Our Theorem \ref{thm-whyburn} is essentially a corollary of the previous theorem of Whyburn.

\begin{theorem}\label{thm-whyburn} Let $\gamma_k$, $k \in \N$, denote a null sequence of Jordan curves in $\Sp^2$. Then there exists a homeomorphism $F:\Sp^2 \to \Sp^2$ such that each curve $F(\gamma_k)$ is a round  circle.
\end{theorem}

\begin{proof} We may assume that the union of Jordan curves $\gamma_k$ is dense in $\Sp^2$ (we can always replace the sequence $\gamma_k$ with a new null sequence of disjoint Jordan curves $\gamma'_m$ such that the sequence $\gamma'_m$ contains $\gamma_k$ as a subsequence and the union of curves $\gamma'_m$ is dense in $\Sp^2$).

Let $S$ be a connected component of 
$$
\Sp^2 \setminus \bigcup_{k}\gamma_k.
$$
As usual, by $\overline{S}$ we denote the closure of $S$. Then the complement of $\overline{S}$ is a union of open Jordan discs (with mutually disjoint closures) that are bounded by some subset of the curves from the sequence $\gamma_k$. Moreover, the union of these Jordan discs is dense in $\Sp^2$ (this is because the union of the corresponding  closed discs contains all the curves $\gamma_k$ and their union is dense in $\Sp^2$). 
By definition, the set $\overline{S}$ is an $S$-curve.

By a triod we mean a plane continuum $K$ which contains a sub-continuum $K_1$ such that $K \setminus K_1$ has at least three components. The letter $Y$ (by this we mean a topological space homeomorphic to the letter $Y$) is an example of a triod.
Recall that the Moore's Triode Theorem \cite{moore} says that one can not pack more than countably many disjoint copies triods in $\Sp^2$,

Each  connected component $S$ contains a copy of the letter $Y$. It is easily seen that the ``interior" of  the standard Sierpinski curve curve contains a copy of the letter $Y$ (by the ``interior" of an $S$ curve we mean the points in the curve that do not lie on an outer boundary circle). Since $\overline{S}$ is homeomorphic to the standard $S$-curve we have that the similar conclusion holds for $S$. 
It follows from the   Moore's Triode Theorem there are at most countably many connected components of $\Sp^2 \setminus (\cup_{k}\gamma_k)$.  On the other hand, we can not cover $\Sp^2$ by finitely many $S$-curves and thus there are infinitely many connected components of  $\Sp^2 \setminus (\cup_{k}\gamma_k)$. We enumerate them as $S_n$, $n \in \N$. We rename the corresponding $S$-curves as $D_n=\overline{S}_n$.

The sequence $D_n$ has the following properties:

\begin{enumerate}
\item The relation 
$$
\bigcup_{n \in \N} D_n=\Sp^2,
$$
holds.
\item If $D_i \cap D_j \ne \empty$ and $i \ne j$,  then $D_i \cap D_j=\gamma_l$, for some $l \in \N$, and $\gamma_l$ is an outer boundary curve for both $D_i$ and $D_j$. Moreover, $D_i$ and $D_j$ lie on the opposite sides of $\gamma_l$.
\end{enumerate}
The above properties (1) and (2) follow from the following observations.
By definition we have the equality
$$
\bigcup_{n \in n} S_n \cup \bigcup_{k \in \N} \gamma_k=\Sp^2.
$$
Moreover, given any $\gamma_k$ there are exactly two $S$-curves $D_i$ and $D_j$ that contain $\gamma_k$ as an outer boundary curve (and they lie on the opposite sides of $\gamma_k$).
Thus
$$
\bigcup_{n \in n} D_n=\bigcup_{n \in n} S_n \cup \bigcup_{k \in \N} \gamma_k=\Sp^2,
$$
and this proves (1). On the hand, if $\gamma$ is an outer boundary curve of some $D_n$ then $\gamma$ is a curve from the sequence $\gamma_k$. This confirms the above property (2).

\vskip .3cm
Next, by $\delta_n$ we denote a sequence of disjoint round circles in $\Sp^2$ whose union is dense in $\Sp^2$. By $E_m$ we denote the sequence of $S$-curves that are closures of the connected components of  $\Sp^2 \setminus (\cup_{k}\delta_k)$.
We construct the map $F$ as follows. 

Let $n_1=k_1=1$. By Theorem \ref{thm-whyburn-1} there exists a homeomorphism $F_1:D_{n_{1}} \to E_{k_{1}}$. Choose an $S$-curve from the sequence $D_n$ that shares a common outer boundary curve $\gamma$ with $D_{n_{1}}$. We relabel this new $S$-curve 
as $D_{n_{2}}$. The corresponding $S$-curve from the sequence $E_n$ is denoted by $E_{k_{2}}$. Then $E_{k_{2}}$ and $E_{k_{1}}$ share the outer boundary curve $F_1(\gamma)$). Let  $F_2:D_{n_{2}} \to E_{k_{2}}$ be a homeomorphism that agrees with $F_1$ 
on $\gamma$ (such $F_2$ exists by Theorem \ref{thm-whyburn-1}). 

We repeat this procedure and construct the sequence of homeomorphisms  $F_j:D_{n_{j}} \to E_{k_{j}}$ such that any two $F_j$ and $F_i$ agree on $D_{n_{j}} \cap D_{n_{i}}$.  Pasting $F_j$'s together gives the required homeomorphism  $F:\Sp^2 \to \Sp^2$.

\end{proof}

\end{document}